\documentclass[a4paper,10pt]{article}
\usepackage{amsmath}
\usepackage{amsfonts}
\usepackage{amssymb}
\usepackage{amscd}
\usepackage{pb-diagram}
\usepackage{color}
\usepackage[all]{xy}
\usepackage{graphicx}
\usepackage{url}
\usepackage{color}
\usepackage{amsthm}
\author{Mathieu Molitor
\\ \it\small{Chaire des structures alg\'ebriques et g\'eom\'etriques}
\\ \it\small{Facult\'e des sciences de base -- Institut de math\'ematiques B}
\\ \it\small{Ecole Polytechnique F\'ed\'erale de Lausanne}
\\ \small{\it{e-mail:}}\,\,\url{pergame.mathieu@gmail.com}
}
\title{Generalization of Hasimoto's transformation}

\date{}

\begin{document}
\newtheorem{lemma}{Lemma}[section]
\newtheorem{definition}[lemma]{Definition}
\newtheorem{proposition}[lemma]{Proposition}
\newtheorem{corollaire}[lemma]{Corollaire}
\newtheorem{theoreme}[lemma]{Theorem}
\newtheorem{remark}[lemma]{Remark}

\maketitle

\begin{abstract}
In this paper, we generalize the famous Hasimoto's transformation by showing that the dynamics of a closed unidimensional vortex filament embedded in a three-dimensional manifold $M$ of constant curvature gives rise under Hasimoto's transformation to the non-linear Schr\"odinger equation.

We also give a natural interpretation of the function $\psi$ introduced by Hasimoto in terms of moving frames associated to a natural complex bundle over the filament.
\end{abstract}

\section*{Introduction}	
The classical vortex filament equation describes the dynamics of a time-dependant ``filament" $\alpha_{t}=\alpha\in Emb(S^{1},\mathbb{R}^{3})$ ($Emb(S^{1},\mathbb{R}^{3})$ being the space of smooth embbedings) and is given by 
\begin{equation}
\dfrac{d\alpha}{dt}=\kappa\cdot B\,,\label{VFE}
\end{equation}
where $\kappa\in C^{\infty}(S^{1},\mathbb{R}^{3})$ and $B$ are respectively the curvature and the binormal of $\alpha$ (see for example \cite{Berger})\,. 
The terminology comes from fluid mechanics; a filament has to be thought of as the very heart of a vortex or a whirlwind, i.e.,  a region of $\mathbb{R}^{3}$ where all the ``vorticity" is concentrated. Its dynamics are actually derived from the so-called ``LIA" approximation (localized induction approximation, see \cite{Kambe}) and it can be shown that Eq. \eqref{VFE} is equivalent to the ``equation of gas dynamics" as well as the``Heisenberg magnetic chain equation" (see \cite{Arnold-Khesin}). See also \cite{Ricca} for a very nice historical survey of the vortex filament equation.\\
Probably one of the most striking features of the vortex filament equation is given by the following. Hasimoto noticed in \cite{Hasimoto}, that the function $\Psi\,:\,S^{1}\rightarrow \mathbb{R}^{3}$ which is defined -- up to a phase factor -- by $\psi(s):=\kappa(s)\cdot e^{i\int_{0}^{s}\,\tau(x)dx}\,,$ where $\kappa$ and $\tau$ are respectively the curvature and the torsion of a filament $\alpha$ solution of Eq. \eqref{VFE}, satifies the non-linear Schr\"odinger equation :
\begin{equation}
-i\,\dfrac{\partial\,\Psi}{\partial t}=\dfrac{\partial^{2}\,\Psi}{\partial s^{2}}+\dfrac{1}{2}\,\vert\Psi\vert^{2}\cdot\Psi\,.
\end{equation} 
This is a remarkable observation  which gave rise to numerous papers mostly on integrable systems since the non-linear Schr\"odinger equation is well known to be a completely integrable system (see \cite{Langer-Perline,Zakharov-Shabat}).\\
The purpose of this paper is twofold and the discussion is divided into two sections. In the first one, an interpretation of the somehow puzzling function $\psi$ introduced without justifications by Hasimoto in \cite{Hasimoto}, is given. More precisely, it is shown that the function $\psi$ can be seen as the infinitesimal rotation of a natural moving frame associated to a natural complex bundle over the filament. In the second section, it is shown that Hasimoto's observation still ``holds" for a closed unidimensional filament embedded in a three-dimensional oriented Riemannian manifold of constant curvature, generalizing the case of $\mathbb{R}^{3}\,.$
\section{Interpretation of the Function $\psi$}
Let $E\overset{\pi^{E}}{\longrightarrow} M$ be a $\mathbb{K}$-vector bundle over a manifold $M$ ($\mathbb{K}=\mathbb{R}$ or $\mathbb{C}$) and let $h^{E}$ be, according to $\mathbb{K},$ an Euclidean or Hermitian structure on $E.$ 
Recall that the orthonormal frame bundle $\mathcal{F}E$ associated to the vector bundle $E$ is defined as the disjoint union $\underset{x\in M}{\cup}\mathcal{F}E_{x}$ where $\mathcal{F}E_{x}$ is the set of all orthonormal frames for the fiber $(E_{x},h^{E}\,)\,.$ It is well known that $\mathcal{F}E$ is a $G$-principal bundle over $M$ with structure group $G$ equal to $O(k)$ if $\mathbb{K=R}$ or $U(k)$ if $\mathbb{K}=\mathbb{C}$ where $\text{dim}_{\mathbb{K}}(E_{x})=k$ for all $x\in M\,.$ 
We shall write $G\hookrightarrow \mathcal{F}E \overset{\pi^{\mathcal{F}E}}{\longrightarrow} M\,.$\\
 $\text{}$\\
Let us give the following useful lemma which describes tangent vectors on $\mathcal{F}E\,.$ 
\begin{lemma}\label{lemma 1}\label{lemme vecteur tangents sur fibre des reperes}
Let $\nabla^{E}$ be a connection on the vector bundle $E$ compatible with $h^{E}$ and let $A$ and $\widetilde{A}$ be two smooth curves of $\mathcal{F}E$ such that $A(t_{0})=\widetilde{A}(t_{0})\,.$ Then we have the following equivalence :
\begin{eqnarray}
&&\dfrac{d}{dt}\bigg\vert_{t_{0}}A(t)=\dfrac{d}{dt}\bigg\vert_{t_{0}}\widetilde{A}(t)
\Leftrightarrow
\left\lbrace
\begin{array}{ccc}
\dot{\alpha}(t_{0})=\dot{(\widetilde{\alpha})}(t_{0})\nonumber\\
\text{and}\\
\text{}\,\,\,\,\,\nabla^{E}_{\dot{\alpha}(t_{0})}A^{j}(t)=\nabla^{E}_{\dot{(\widetilde{\alpha})}(t_{0})}(\widetilde{A})^{j}(t)\nonumber\\
\text{for all}\,\,\,j\in\{1,...,k\}\,,
\end{array}
\right.
 \end{eqnarray}
where $\alpha(t):=\pi^{\mathcal{F}E}(A(t))$ and where we write $A(t)=\{A^{1}(t),...,A^{k}(t)\}\,,$ $A^{i}(t)$ being an element of $E_{\alpha(t)}$ for all $i\in \{1,...,k\}\,.$
\end{lemma}
\begin{proof} 
	Lemma \ref{lemma 1} can be proved easily using local charts, as given for example in \cite{Naber}, Section 3.3.
\end{proof}
Using Lemma \ref{lemma 1}, it is easy to define a connection form $\theta^{E}\in \Omega^{1}(\mathcal{F}E,\mathfrak{g})$ (here $\mathfrak{g}$ stands for $\mathfrak{o}(k)$ or $\mathfrak{u}(k)$) on the principal bundle $G\hookrightarrow \mathcal{F}E \overset{\pi^{\mathcal{F}E}}{\longrightarrow} M$ as follows :
\begin{equation}\label{definition connection de ehresmann}
\theta^{E}_{A(0)}\Big(\dfrac{d}{dt}\bigg\vert_{0}\,A(t)\Big):=\Big(h^{M}_{\alpha(0)}\big(\nabla_{\overset{\cdot}{\alpha}(0)}A^{j}(t),A^{i}(t)\big)\Big)_{1\leq i,j\leq k}\,,
\end{equation}
where $A(t)=\{A^{1}(t),...,A^{k}(t)\}$ is a smooth curve of $\mathcal{F}E$ and where $\alpha(t):=\pi^{\mathcal{F}E}\big(A(t)\big)\,.$
The connection $\theta^{E}$ is the \textbf{Ehresmann connection} and is well known in Riemannian geometry (see \cite{Kobayashi-Nomizu,Kobayashi-Nomizu-2,Petersen,Spivak2}...). In a certain sense, this connection measures the infinitesimal rotation of a moving frame over $M\,.$ \\\\
Let us now consider a filament $\Sigma:=\alpha(S^{1})$ embedded in a three-dimensional oriented Riemannian manifold $(M,h^{M})$ (here $\alpha\,:\,S^{1}\rightarrow M$ is an embedding). Assume furthermore that for all $s\in S^{1},$ $(Trace\,\Pi_{\Sigma})(\alpha(s))\neq 0$ where $Trace\,\Pi_{\Sigma}$ denotes the trace of the second fundamental form $\Pi_{\Sigma}$ of the submanifold $\Sigma\subseteq M\,.$ This extra assumption allows us to define at any point $\alpha(s)$ of the filament $\Sigma\,,$ its associated Frenet frame :
\begin{eqnarray}
\alpha^{Fr}(s):=(T,N,B)\in \big(\mathcal{F}TM\vert_{\Sigma}\big)_{\alpha(s)}\,\,\,\text{where}\,\,\,
\left\lbrace
\begin{array}{ccc}
T&:=&\dfrac{\overset{\cdot}{\alpha}(s)}{\Vert \overset{\cdot}{\alpha}(s)\Vert}\,,\\
N&:=&\dfrac{\big(Trace\,\Pi_{\Sigma}\big)(\alpha(s))}{\Vert\big(Trace\,\Pi_{\Sigma}\big)(\alpha(s)) \Vert}\,,\\
B&:=&T\times N\,.
\end{array}
\right.\nonumber
\end{eqnarray}
Here $TM\vert_{\Sigma}:=j^{*}_{\Sigma}TM$ with $j_{\Sigma}\,:\,\Sigma\hookrightarrow M$ being the canonical inclusion. We thus get a smooth curve $\alpha^{Fr}\,:\,S^{1}\rightarrow \mathcal{F}TM\vert_{\Sigma}$ which gives rise to another smooth curve $\alpha^{\mathfrak{o}(3)}\,:\,S^{1}\rightarrow \mathfrak{o}(3)$ via the formula :
\begin{eqnarray}
\alpha^{\mathfrak{o}(3)}(s_{0}):=\theta^{TM}_{\alpha^{Fr}(s_{0})}\Big(\dfrac{d}{ds}\bigg\vert_{s_{0}}\,\alpha^{Fr}(s)\Big)\,,
\end{eqnarray}
for $s_{0}\in S^{1}\,.$ In particular, if the filament $\alpha$ is parameterized by arclength, then it is easy to see using Eq. \eqref{definition connection de ehresmann} that $\alpha^{\mathfrak{o}(3)}(s)$ takes the form 
\begin{eqnarray}\label{equation formules de frnet}
\alpha^{\mathfrak{o}(3)}(s)=
\begin{pmatrix}
0&-\kappa(s)&0\\
\kappa(s)&0&-\tau(s)\\
0&\tau(s)&0
\end{pmatrix}
\end{eqnarray}
where $\kappa,\tau\,:\,S^{1}\rightarrow \mathbb{R}$ are two smooth functions on $S^{1}\,.$ In the case where $M=\mathbb{R}^{3}$ endowed with the canonical metric and orientation, then Eq. $\eqref{equation formules de frnet}$ corresponds to the usual Frenet Formulas and $\kappa$ and $\tau$ are respectively the curvature and the torsion of $\Sigma$ (see \cite{Berger}, chapitre 8).\\\\
More generally, if we are given a $\mathbb{K}$-vector bundle over $\Sigma$ of rank $k$ with an Euclidean (or Hermitian) structure on it, a compatible connection and a moving frame\footnote{By ``moving frame", we mean a smooth curve, or more precisely a loop, of the associated orthonormal frame bundle.}, it is then possible to associate an element of $L\mathfrak{o}(k):=C^{\infty}(S^{1},\mathfrak{o}(k))$ (or $L\mathfrak{u}(k)$)\,.\\
In this spirit, the simplest vector bundle over $\Sigma$ which takes account of the geometry of the normal bundle $N\Sigma$ of $\Sigma$ in $M$ is surely given by the complex line bundle $E\overset{\pi^{E}}{\longrightarrow}\Sigma$ whose fiber at a point $\alpha(s)\in \Sigma$ is given by 
\begin{eqnarray}
E_{\alpha(s)}:=\text{Vect}_{\mathbb{C}}\{N+iB\}\subseteq T_{\alpha(s)}M^{\mathbb{C}}\,,
\end{eqnarray}
where the space $T_{\alpha(s)}M^{\mathbb{C}}$ denotes the complexification of $T_{\alpha(s)}M\,.$ Denoting $J^{\mathbb{C}}\,:\,(N\Sigma)_{\alpha(s)}^{\mathbb{C}}\rightarrow (N\Sigma)_{\alpha(s)}^{\mathbb{C}}$ the $\mathbb{C}-$extension of $J$ on the complexification $(N\Sigma)_{\alpha(s)}^{\mathbb{C}}$ of $(N\Sigma)_{\alpha(s)}\,,$ we note that $N+iB$ is an eigenvector of $J^{\mathbb{C}}$ with corresponding eigenvalue $-i\,.$ The metric $h^{M}$ and the associated Levi-Civita connection $\nabla$ naturally induce a Hermitian structure $h^{E}$ on $E$ and a compatible connection $\nabla^{E}\,.$ Thus, according to Lemma \ref{lemme vecteur tangents sur fibre des reperes}, we get a connection form $\theta^{E}\in \Omega^{1}(\mathcal{F}E,\mathfrak{u}(1))\,.$\\\\
Let us consider a smooth curve $A\,:\,S^{1}\rightarrow \mathcal{F}E$ of $\mathcal{F}E$ such that $\pi^{E}(A(s))=\alpha(s)$ for all $s\in S^{1}\,.$ The curve $A$ is necessarily of the form 
\begin{eqnarray}
A(s)=\bigg\{e^{i\rho(s)}\dfrac{N+iB}{\sqrt{2}}\bigg\}
\end{eqnarray}
for all $s\in S^{1},$ where $\rho\,:\,S^{1}\rightarrow \mathbb{R}$ is a smooth map. For $s_{0}\in S^{1}\,,$ an easy calculation shows that
\begin{eqnarray}\label{equation rotation instantannee pour un seul vecteur}
\theta^{E}_{A(s_{0})}\Big(\dfrac{d}{ds}\bigg\vert_{s_{0}}\,A(s)\Big)=\Big(i\big(\dot{\rho}(s_{0})-\tau(s_{0})\big)\Big)\in \mathfrak{u}(1)\,.
\end{eqnarray} 
In particular, the curve $A$ has zero infinitesimal rotation if and only if $\dot{\rho}-\tau=0\,,$ i.e. if $\rho(s) =\int_{0}^{s}\,\tau(x) dx$ (modulo an additive contant). \\\\
Now, in order to consider a vector bundle which also takes into account the ``tangential" geometry of the filament $\Sigma\,,$ it is natural to consider the complex vector bundle $F\overset{\pi^{F}}{\longrightarrow}\Sigma$ whose fiber at a point $\alpha(s)\in \Sigma$ is defined by 
\begin{eqnarray}
F_{\alpha(s)}:=\text{Vect}_{\mathbb{C}}\{T,N+iB\}\subseteq T_{\alpha(s)}M^{\mathbb{C}}\,.
\end{eqnarray}
Again, we get a Hermitian structure $h^{M}\,,$ a compatible connection $\nabla^{F}$ and a connection form $\theta^{F}\in \Omega^{1}(\mathcal{F}F,\mathfrak{u}(2))\,.$ In view of the above, let us denote by $B\,:\,S^{1}\rightarrow \mathcal{F}F$ the map defined as 
\begin{eqnarray}
B\big(\alpha(s)\big):=\bigg\{T,\big(e^{i\int_{0}^{s}\,\tau(x) dx}\big)\dfrac{N+iB}{\sqrt{2}}\bigg\}\,,
\end{eqnarray}
for $s\in S^{1}\,.$\\
Again, a simple calculation shows that 
\begin{eqnarray}
\theta^{F}_{B(s_{0})}\Big(\dfrac{d}{ds}\bigg\vert_{s_{0}}\,B(s)\Big)=\dfrac{1}{\sqrt{2}}\cdot
\begin{pmatrix}
0&-\psi\\
\overline{\psi}&0
\end{pmatrix}\in \mathfrak{u}(2)\,,
\end{eqnarray}
where $\psi$ is the famous Hasimoto's function, i.e. $\psi(s)=\kappa(s)\cdot e^{i\int_{0}^{s}\,\tau(x) dx}\,.$ Hence, the function $\psi$ of Hasimoto measures the infinitesimal rotation of the natural moving frame $B$ over 
$\Sigma\,.$
\section{Generalization of Hasimoto's Transformation for Manifolds of Constant Curvature}
Let $(M,g)$ be a three-dimensional oriented Riemannian manifold and $\alpha_{t}\,:\,S^{1}\rightarrow M$ a time-dependent embedding. For $t\in(-\epsilon, \epsilon)\,,$ we assume that
\begin{center}
\begin{description}
\item[$\bullet$] $(Trace\,\Pi_{\Sigma_{t}})\big(\alpha_{t}(s)\big)\neq 0\,,$ for all $s\in S^{1}$ and for all $t\in (-\varepsilon,\varepsilon)$ (here $\Sigma_{t}:=\alpha_{t}(S^{1})$)\,,\\
\item[$\bullet$] $\alpha_{t}$ is parameterized by arclength, i.e., $\Vert \frac{d}{ds}\,\alpha_{t}(s)\Vert=1$ for all $t\in (-\varepsilon,\varepsilon)$ and for all $s\in S^{1}\,,$\\
\item[$\bullet$] $\alpha_{t}$ is a solution of the vortex filament equation, i.e., $\frac{d}{dt}\,\alpha_{t}=\kappa\cdot B\,.$
\end{description}
\end{center}
\begin{remark}
The first condition above ensures that the Frenet frame associated to the curve $\alpha_{t}$ exists for all $t\in (-\epsilon, \epsilon)\,.$
\end{remark}
\begin{remark}
If $\alpha_{t}$ is a solution of the vortex filament equation which is parameterized by arclength for $t=0\,,$ then one can show that $\alpha_{t}$ is also parameterized by arclength for all $t\,.$ 
\end{remark}

We will now give the evolution's equation of the curvature $\kappa$ and the torsion $\tau$ of the curve $\alpha_{t}\,.$ Recall that the curvature and the torsion of a curve are defined by Eq. \eqref{equation formules de frnet}. Consequently, we have the Frenet formulas : 
\begin{eqnarray}\label{hummm}
\dfrac{D}{\partial s}T=\kappa\cdot N\,;\,\,\,\dfrac{D}{\partial s}N=-\kappa\cdot T+\tau\cdot B\,;\,\,\,\dfrac{D}{\partial s}B=-\tau\cdot N\,,
\end{eqnarray}
where $\frac{D}{\partial s}$ denotes the covariant derivative along the filament $\alpha_{t}\,.$ 
Using Eq. \eqref{hummm} and proceeding exactly as in \cite{Brylinski}, we find out that the evolution equations of $\kappa$ and $\tau$ are :
\begin{eqnarray}
\left\lbrace
\begin{array}{ccc}\label{evolution}
\underset{\text{}}{\dfrac{\partial\,\kappa}{\partial t}}&=&-2\tau\dfrac{\partial\,\kappa}{\partial s}-\kappa\,\dfrac{\partial\,\tau}{\partial s}+\kappa\cdot\big(T,B,T,N\big)\,,\\
\dfrac{\partial\,\tau}{\partial t}&=&\dfrac{\partial}{\partial s}\bigg(\dfrac{1}{\kappa}\dfrac{\partial^{2}\,\kappa}{\partial s^{2}}-\tau^{2}+\dfrac{1}{2}\kappa^{2}+\big(T,B,T,B\big)\bigg)-\kappa\cdot\big(B,T,N,B\big)\,,
\end{array}
\right.
\end{eqnarray}
where $(\,.\,,\,.\,,\,.\,,\,.\,):=h^{M}\big(R(\,.\,,\,.\,)\,.\,,\,.\,\big)\,,$ with $R\,:\,\mathcal{X}(M)\times\mathcal{X}(M)\times\mathcal{X}(M)\rightarrow \mathcal{X}(M)$ denoting the Riemannian curvature.  \\\\
Under the additional assumption that the curvature of $M$ is constant, then the curvature tensor of $M$ reduces to (see \cite{Do-carmo}, Lemma 3.4 page 96) :
\begin{eqnarray}
\big(X,Y,W,Z\big)=K_{0}\Big(h^{M}(X,W)\cdot h^{M}(Y,Z)-h^{M}(X,Z)\cdot h^{M}(Y,W)\Big)\,,
\end{eqnarray}
for $X,Y,W,Z\in \mathcal{X}(M)$ and a certain constant $K_{0}\,.$ In particular, we have $\big(T,B,T,B\big)=K_{0}$ and $\big(B,T,N,B\big)=\big(T,B,T,N\big)=0\,.$\\
In this context, i.e., assuming $\kappa$ and $\tau$ associated to a curve $\alpha_{t}$ solution of the vortex filament equation, let us consider 
$\psi_{t}\,:\,S^{1}\rightarrow\mathbb{C},\,s\mapsto \kappa\cdot e^{i\,\int_{0}^{s}\,\tau(x)dx}\,.$ Exactly as in the proof of Theorem 3.5.8 in \cite{Brylinski}, 
and using Eqs. \eqref{evolution}, a direct calculation shows that
\begin{eqnarray}
-i\,\dfrac{\partial\,\psi}{\partial t}=\dfrac{\partial^{2}\,\psi}{\partial s^{2}}+\dfrac{1}{2}\,\vert\psi\vert^{2}\cdot\psi-\psi\cdot A(t)\,,
\end{eqnarray}
where $A(t):=\big(\frac{1}{\kappa}\frac{\partial^{2}\,\kappa}{\partial s^{2}}-\tau^{2}+\frac{1}{2}\kappa^{2}\big)(s)\,\big\vert_{s=0}\,.$ Set $\Psi_{t}(s):=e^{i\int_{0}^{t}\,A(x)dx}\cdot \psi_{t}(s)$ for $t\in (-\varepsilon,\varepsilon)$ and $s\in S^{1}\,.$ Again,  a direct calculation shows that $\Psi_{t}$ satisfies the non-linear Schr\"odinger equation :
\begin{eqnarray}
{-i\,\dfrac{\partial\,\Psi}{\partial t}=\dfrac{\partial^{2}\,\Psi}{\partial s^{2}}+\dfrac{1}{2}\,\vert\Psi\vert^{2}\cdot\Psi}\,.
\end{eqnarray}
\section*{Acknowledgments} I would like to give special thanks to Tilmann Wurzbacher for his careful and critical readind of the ``french version" of this paper (i.e. the corresponding part of my thesis).\\
This work was done with the financial support of the Fonds National Suisse de la Recherche Scientifique under the grant PIO12--120974/1.

\end{document}